\documentclass[letterpaper, 10 pt, conference]{ieeeconf}  % Comment this line out if you need a4paper

\IEEEoverridecommandlockouts                          
\usepackage{amsmath}
\usepackage{amssymb}
\usepackage{comment}
\usepackage{color}
\usepackage[normalem]{ulem}
\newcommand{\ZL}{\color{red}}

\newtheorem{thm}{Theorem}
\newtheorem{cor}[thm]{Corollary}
\newtheorem{lem}[thm]{Lemma}
\newtheorem{prop}[thm]{Proposition}
\newtheorem{defn}[thm]{Definition}
\newtheorem{rem}[thm]{Remark}
\newtheorem{example}[thm]{Example}
\newtheorem{assumption}{Assumption}

\newcommand{\R}{\ensuremath{\mathbb R}}    % Reelle Zahlen
\newcommand{\N}{\ensuremath{\mathbb N}}    % Nat"urliche Zahlen
\newcommand{\iso}{\circ}
\newcommand{\<}{\langle}
\renewcommand{\>}{\rangle}

\newcommand{\calA}{\mathcal A}

\newcommand{\calK}{\mathcal K}

\newcommand{\eps}{\varepsilon}

\newcommand{\bmat}[4]
{
	\begin{bmatrix}
		#1 & #2\\
		#3 & #4
	\end{bmatrix}
}
\newcommand{\bvec}[2]
{
	\begin{bmatrix}
		#1\\
		#2
	\end{bmatrix}
}
\newcommand{\sbmat}[4]{\left[\begin{smallmatrix}#1 & #2\\#3 & #4\end{smallmatrix}\right]}
\newcommand{\sbvec}[2]{\left[\begin{smallmatrix}#1\\#2\end{smallmatrix}\right]}
\newcommand{\wt}{\widetilde}
\title{\LARGE \bf 
Strict Dissipativity and turnpike for LQ Optimal Control Problems with Possibly Boundary Reference
}
\author{Zhuqing Li and Roberto Guglielmi%
\thanks{*The authors acknowledge the support of the Natural Sciences and Engineering Research Council of Canada (NSERC), funding reference number RGPIN-2021-02632}% <-this % stops a space
\thanks{$^{1}$Zhuqing Li is with the Department of Mechanical and Aerospace Engineering, University of California San Diego,
        La Jolla, CA 92093, USA
        {\tt\small zhl160@ucsd.edu}}%
\thanks{$^{2}$Roberto Guglielmi is with the Department of Applied Mathematics, University of Waterloo,
        Waterloo, ON N2L 3G1, Canada
        {\tt\small roberto.guglielmi@uwaterloo.ca}}%
}

\begin{document}

\maketitle
\thispagestyle{empty}
\pagestyle{empty}

\begin{abstract}
In this paper we investigate the turnpike property for constrained LQ optimal control problem in connection with dissipativity of the control system. We determine sufficient conditions to ensure the turnpike property in the case of a turnpike reference possibly occurring on the boundary of the state constraint set.
\end{abstract}

\section{Introduction}
Turnpike phenomena concern qualitative and quantitative properties of optimally controlled dynamics. More precisely, turnpike refers to the property of optimal solutions to Optimal Control Problems (OCPs) to approach a prescribed reference of the control system and stay close to it during most of the time evolution. Turnpike phenomena were first reported by Ramsey \cite{Ramsey} in 1928 and von Neumann \cite{VonNeumann} in 1945 in the context of mathematical economics. Dorfman et al. \cite{Dorfman} first used the term turnpike for such property. Since then, turnpike phenomena have received continuous interest in economics literature because of the structural insights they provide on the structure of the optimal solutions~\cite{Samuelson,Wilde,McKenzie,McKenzie1963}.
% on this subject concern specific problems arising in economics by exploiting the structure of the optimality conditions. 
%Following those results, turnpike phenomena have been widely observed and investigated in different fields, e.g., mathematical biology \cite{Bio} and chemical process \cite{Che}.
Turnpike properties have been extensively studied by the mathematical control community, investigating several sufficient conditions to ensure such behavior. %~\cite{Anderson,Carlson,Carlsonbook,Larsexp,Lars2016,LouHongwei,Zua2013,Rapaport,Zuanonlinear}. 
We refer to \cite{Zaslavski06,Zaslavski14,Larssurvey} for an overview on turnpike properties for various control systems and a comprehensive literature review. The applications of turnpike are multifold. For instance, turnpike theory has been used as a method of synthesizing long-term optimal trajectories \cite{Larssyn,Zua2016,Riccati}, analyzing Model Predictive Control (MPC) schemes \cite{LarsMPC2016,LarsMPC2017} and shape optimization in aircraft design \cite{Lance}. 
%In recent years, various turnpike properties have been considered in specific problems like membrane-filtration systems \cite{Membrane} and control of chemical reactors with uncertain models \cite{Chemicaluncertain}. 
In the last decades, the interplay between economic MPC~\cite{Rawlings} and turnpike properties has received great attention and led to the study of the connection between turnpike and dissipativity. 

Dissipativity notions, as introduced by Willems \cite{Willems1,Willems2}, naturally characterize an abstract energy balance of the dynamical system in terms of a storage function and a supplied energy. In~\cite{Diehl}, dissipativity is exploited to construct a Lyapunov function for the economic MPC-closed loop. Later in \cite{Angeli}, the authors prove the equivalence between strong duality and strict dissipativity with a linear storage function, which allows to use the same Lyapunov function in a more general nonlinear setting. In both cases, the running (or stage) cost in the optimal control problem was used as the supply rate in the dissipativity formulation. This idea can be adapted to investigate the sufficient conditions for turnpike properties by adding a dissipation rate in the supply rate. Some early results in this direction can be found in \cite{Willemsearly,Carlsonbook}. Then in the paper \cite{Lars2013}, it was observed that strict dissipativity plus a suitable controllability condition imply turnpike. Following this result, in \cite{Lars2016}, several equivalence statements between strict dissipativity properties
and the turnpike properties were proposed in the nonlinear setting. In recent papers \cite{LarsGugdis,LarsGugcont}, an if-and-only-if characterization of the so-called measure turnpike property in terms of the stabilizability and detectability for unconstrained Linear Quadratic (LQ) OCPs was deduced in discrete- and continuous-time, respectively. 

%Generally speaking, the characterization of the turnpike for unconstrained problems has become mature, but some open problems still exist in different settings. For one thing, although several necessary conditions of the turnpike property have been deduced for unconstrained LQ optimal control problem (\cite{LarsGugdis,LarsGugcont}), even in infinite dimensional setting (\cite{necessary}), one does not know any interesting necessary condition for the turnpike in constrained case. For another, the papers concerned with the turnpike property for constrained problem generally assume the turnpike reference lies in the interior of the constrained set, but the turnpike reference may naturally occur on the boundary of the constrained set. As commented in the survey \cite{Larssurvey}, it is currently an open question whether this case can be addressed by dissipitivity based method.

{
Despite the characterization of the turnpike in terms of structural control properties of the system is by now well-understood for unconstrained problems in both finite~\cite{LarsGugdis,LarsGugcont} and infinite dimensional~\cite{necessary} settings, many questions are still open in the constrained case. In particular,~\cite{LarsGugdis,LarsGugcont} provide sufficient conditions for turnpike property for constrained problems, assuming that the turnpike reference lies in the interior of the constraint set. However, the turnpike reference may occur on the boundary of the constraint set, or even depend on the choice of the constraint set itself. As noted in~\cite{Larssurvey}, it is currently an open question whether this case can be addressed by dissipitivity-based method.}

Our work addresses this issue for a class of LQ OCPs with a convex constraint set. To the best of our knowledge, the paper presents the first sufficient condition that ensures the occurrence of the turnpike property with turnpike reference possibly on the boundary of the constraint set. {Our method relies on the dissipitivity approach in \cite{LarsGugdis} via a perturbed estimate for the constrained optimal steady state, %, which is defined as the minimizer of the corresponding constrained optimal steady state problem, 
and makes use of the KKT conditions to deduce the so-called strict (pre-)dissipativity of the constrained OCP. In combination with a generalized stabilizability assumption, this eventually ensures the turnpike property. Thanks to a characterization of strict dissipativity in terms of detectability of the system, we are thus able to show the turnpike property under classical} structural properties of the OCP.

Our paper is organized as follows. In Section \ref{setting}, we describe the mathematical setting of the constrained LQ OCP and introduce the notions of dissipativity and turnpike. % notions to be considered in the paper. 
%Although in the paper we adopt a discrete-time framework, we believe analogous results hold also in the continuous-time setting. 
In Section \ref{existing} we show the existence and uniqueness of solutions to the constrained optimal steady state. %, then we briefly discuss the existing results on the sufficient conditions of the measure turnpike property. 
In Section~\ref{mainresults}, we state and prove our main result, which shows that strict dissipativity and a stabilizability condition imply the measure turnpike property, where the turnpike reference may lie on the boundary of the constraint set. We conclude the paper with two illustrative examples.

\section{Mathematical setting}\label{setting}
We consider the following constrained LQ optimal control problem $(CLQ)_N$ in discrete-time
\begin{equation}\label{cost}
\mathop{\rm minimize}_{u\in \calA^N(x_0)}\ \displaystyle\Sigma_{i=0}^{N-1} \ell(x(i),u(i)),
\end{equation}
where the stage cost $\ell$ is defined by
\begin{equation}\label{runcost}
\ell(x,u):=x^TQx+u^TKu+2z^Tx+2v^Tu+c
\end{equation}
and the pair $(x(\cdot),u(\cdot))$ is subject to
\begin{equation}\label{dynamic}
\begin{array}{l}
x(i+1)=Ax(i)+Bu(i),\quad i = 0,\ldots,N-1,\\[0.3ex]
x(0)=x_0\in \R^n.
\end{array}
\end{equation}
We assume $N\in\N$, $A\in \R^{n\times n}$, $B\in \R^{n\times m}$, $Q\in \R^{n\times n}$ and $R\in \R^{m\times m}$. Moreover, we further assume that $Q$ is positive semi-definite and $R$ is positive definite, so that there exist $C\in\R^{n\times n}$ and invertible $K\in \R^{m\times m}$ such that $Q=C^TC$ and $R=K^TK$. We denote by $S \subset \R^n\times \R^m$ the constraint set on the state/control pair, and define
\begin{equation*}
X:=\{x\in\R^n\;|\;\text{there exists}\;u\in\R^m\;\text{s.t.}\;(x,u)\in S\}.
\end{equation*}
We say that the pair $(x(\cdot),u(\cdot))$ is admissible if
\begin{equation}\label{constraint}
(x(i),u(i))\in S,\;\;\forall i=0,1,...,N-1,\;\;x(N)\in X.
\end{equation}
We assume that $S$ can be represented as the sublevel set of some convex function, i.e.,
\begin{equation}\label{constraintrepre}
S=\{(x,u)\in\R^{n+m}\;|\;g(x,u)\leq 0\}
\end{equation}
for some convex function $g : \R^n\times \R^m\to\R$, and the boundary of $S$ is characterized by the zero level set of $g$:
\begin{equation*}
\partial S=\{(x,u)\in\R^{n+m}\;|\;g(x,u)=0\}.
\end{equation*}
{Then it is easy to verify that $S$ and $X$ are closed convex sets.} We denote by $\calA^N(x_0)$ (resp. $\calA^\infty(x_0)$) the set of all admissible controls corresponding to time horizon $N$ (resp. $\infty$) and initial condition $x_0\in \R^n$. In other words, $u\in \calA^N(x_0)$ (resp. $\calA^\infty(x_0)$) if and only if the pair $(x(\cdot),u(\cdot))$ satisfies the condition~\eqref{constraint} over the time horizon $0,1,\ldots, N$ (resp. $0,1,\ldots$).
In the remaining part of this paper, we will use $x^*_N(\cdot,x_0)$ and $u^*_N(\cdot,x_0)$ to denote the optimal trajectory and optimal control corresponding to time horizon $N$ and initial condition $x_0$, if the optimal pair exists. {Notice that, if $\calA^\infty(x_0)\neq\emptyset$, then the optimal pair exists and is unique. The proof is much similar to \cite[Theorem 3.1.1]{thesis}.}

%Our main result characterizes turnpike property in terms of simple structural properties of the optimal control problem. To begin with, we introduce 
The notions of dissipativity used in the paper involve the notions of storage functions and dissipation rates. The latter are defined in the class
\begin{align*}
\calK:=\{f:\R^+\to\R^+&\;|\;f\;\mathrm{continuous},\;\\
&\mathrm{strictly\;increasing\;with}\;f(0)=0\}.
\end{align*}
The pair $(x_e,u_e)\in\R^{n+m}$ is called a \emph{steady state} of the constrained control system~\eqref{dynamic}-\eqref{constraint} if $(x_e,u_e)\in S$ and $x_e=Ax_e+Bu_e$.
\begin{defn}
We say that the constrained OCP \eqref{cost}-\eqref{dynamic} is strict pre-dissipative at some steady state $(x_e,u_e)\in S$ if there exists a storage function $V:\R^n\to\R$ which is bounded on bounded subsets of $\R^n$, a norm $\|\cdot\|_{\mathrm{diss}}$ defined on $\R^n\times\R^m$, and a dissipation rate $\alpha\in\calK$ such that
\begin{align*}
V(Ax+Bu)\leq &V(x)+\ell(x,u)-\ell(x_e,u_e)\\
&-\alpha(\|(x-x_e,u-u_e)\|_{\mathrm{diss}}),\;\;\forall (x,u)\in S.
\end{align*}
Moreover, we say that the problem \eqref{cost}-\eqref{dynamic} is strict dissipative at $(x_e,u_e)$ if it is strict pre-dissipative and, in addition, the storage function $E$ is bounded from below on {$X$}. 
\end{defn}
For the notion of turnpike in the discrete-time setting, we recall the following definition from \cite[Definition 2.1]{LarsGugdis}.
\begin{defn}\label{def:turnpike}
We say that the constrained OCP \eqref{cost}-\eqref{dynamic} satisfies the measure turnpike property at some steady state $(x_e,u_e)\in S$ on the initial set $X_{\mathrm{tp}}\subset X$ if, for any $\eps>0$, there exists some constant $N_{\eps}\in\N^+$ such that, for any $x_0\in X_{\mathrm{tp}}$ and $N\in\N^+$, if the optimal pair exists, then
\begin{align*}
\#\big\{ i= 0,\ldots,N-1 \; \big| \; \| &x^*_N(i,x_0)-x_e\|^2\\
&+\| u^*_N(i,x_0)-u_e\|^2 >\eps \big\} \leq N_{\eps},
\end{align*}
where {we have used $\|\cdot\|$ to denote the corresponding $l^2$ norm and $\#$ to denote the cardinality of some finite set.}
\end{defn}
\section{Constrained optimal steady state}\label{existing}
%The characterization of the turnpike reference is a crucial step in the analysis of the turnpike property. To show the existence and uniqueness of the constrained optimal steady state, we need to introduce the definition of unobservable eigenvalues.\\
We now turn to the analysis of the constrained optimal steady state problem, which will provide a characterization of the turnpike reference. We start by introducing the notion of unobservable eigenvalues.
\begin{defn}\label{def:unobs}
We say that $\gamma \in \mathbb{C}$ is an unobservable eigenvalue of the pair $(A,C)$ if {$\gamma$ is an eigenvalue of $A$ and there exists a non-trivial eigenvector $x_0\in\R^n$ corresponding to the eigenvalue $\gamma$ such that $x_0\in\ker C$.}
\end{defn}
\begin{rem}
In the discrete-time setting, an element $x_0\in\R^n\setminus\{0\}$ is unobservable if $CA^ix_0=0$ for all $i\in\N$. However, if $x_0$ is an eigenvector of $A$, then clearly the above condition is equivalent to $x_0\in\ker C$.
\end{rem}
\begin{lem}\label{posidefin}
Assume that each unobservable eigenvalue $\gamma$ of the pair $(A,C)$ satisfies $|\gamma| \neq1$. Then the matrix
\begin{equation*}
\bmat{Q}{0}{0}{R}=\bmat{C^{T}C}{0}{0}{K^{T}K}
\end{equation*}
is positive definite on {the kernel space of the block matrix $[A-I\ B]$.}
\end{lem}
\begin{proof}
Obviously, the matrix $\sbmat{C^{T}C}{0}{0}{K^{T}K}$ is positive semi-definite. Suppose $(x,u)$ is a non-trivial element in $\ker[A-I\ B]$ satisfying
\begin{equation*}
\left\<\sbmat{C^{T}C}{0}{0}{K^{T}K}\bvec{x}{u},\bvec{x}{u}\right\>=\|Cx\|^2+\|Ku\|^2=0.
\end{equation*}
Since $K$ is invertible, we deduce that $u=0$, $Cx=0$, $x\neq 0$, and $x=Ax$. This means that $x$ is an unobservable eigenvalue of $A$ corresponding to eigenvalue $1$, which contradicts our assumption.
\end{proof}
\begin{prop}
Assume that each unobservable eigenvalue~$\gamma$ of the pair $(A,C)$ satisfies $|\gamma| \neq1$. Then the constrained optimal steady state problem
\begin{align}\label{conoss}
\begin{split}
&\mathop{\rm minimize}_{x\in\R^n,u\in\R^m}\;\;\ell(x,u)\\
&\text{s.t.}\;\;(x,u)\in S,\quad Ax+Bu=x
\end{split}
\end{align}
admits a unique minimizer $(x_e,u_e)$.
\end{prop}
\begin{proof}
By Lemma \ref{posidefin}, the quadratic term of $\ell$ is positive definite on $\ker[A-I\ B]$, so $\ell$ is bounded from below on $\ker[A-I\ B]$. Besides, $\ell$ is continuous in $(x,u)$ and $\ell(x,u)\to\infty$ as $\|x\|+\|u\|\to\infty$ on $\ker[A-I\ B]$. So, the minimizer of \eqref{conoss} exists on $\ker[A-I\ B]$. Now assume that $(x_1,u_1)$ and $(x_2,u_2)$ are two distinct minimizers of \eqref{conoss}. Since $S$ is a convex set, $\left(\frac{x_1+x_2}{2},\frac{u_1+u_2}{2}\right)\in S$. Simple calculations show that
\begin{align*}
&\frac{\ell(x_1,u_1)+\ell(x_2,u_2)}{2}-\ell\left(\tfrac{x_1+x_2}{2},\tfrac{u_1+u_2}{2}\right)\\
&\qquad=\frac{1}{4}\left(\|C(x_1-x_2)\|^2+\|K(u_1-u_2)\|^2\right).
\end{align*}
Hence we conclude that $\ell\left(\frac{x_1+x_2}{2},\frac{u_1+u_2}{2}\right)$ attains a strictly lower value than $\ell(x_1,u_1)=\ell(x_2,u_2)$, which leads to a contradiction.
\end{proof}
We refer to the unique minimizer $(x_e,u_e)$ of~\eqref{conoss} as the \emph{constrained optimal steady state} of the constrained dynamical OCP~\eqref{cost}-\eqref{dynamic}, or simply as the optimal steady state for the unconstrained case (i.e., $S=\R^n\times\R^m$). In the sequel of this paper, we assume that the KKT conditions are verified at $(x_e,u_e)$. That is, we assume that there exist some $\lambda\in \R^n$ and $\mu\in\R$ such that
\begin{equation}\label{KKT}
%\addtolength{\jot}{5pt}
\!\!\!\left\{
\begin{alignedat}{2}
&\bvec{2C^TCx_e+2z}{2K^TKu_e+2v}+\bvec{A^T-I}{B^T}\lambda+(\partial g(x_e,u_e))^T\mu\ni 0,\\
&Ax_e+Bu_e=x_e,\\
&\mu\geq 0,\\
&g(x_e,u_e)\mu=0,
\end{alignedat}
\right.
\end{equation}
where $\partial g(x,u)$ denotes the subdifferential of $g$ at $(x,u)$. In general, the KKT conditions can be guaranteed by very loose regularity conditions, e.g., Slater's condition, which requires that the feasible region $S$ contains a steady state in its interior. Here we refer to \cite[Chapter 4]{Lange} for a detailed discussion on KKT conditions and constraint qualifications.

\section{Main results}\label{mainresults}
\begin{thm}\label{dissres}
Assume that each nonobservable eigenvalue satisfies $|\mu| \neq1$, and that the KKT conditions are verified at the constrained optimal steady state $(x_e,u_e)$. Then the constrained OCP~\eqref{cost}-\eqref{dynamic} is strictly pre-dissipative at  $(x_e,u_e)$. In addition, if the pair $(A,C)$ is detectable, then the problem~\eqref{cost}-\eqref{dynamic} is strictly dissipative at $(x_e,u_e)$.
\end{thm}
\begin{proof}
Let $(x_e,u_e)$ be the constrained optimal steady state of problem \eqref{conoss}. Recall from \eqref{KKT} that there exist some multipliers $\lambda\in\R^n$ and $\mu\in\R$ such that
\begin{equation}\label{eq:pfthm}
%\addtolength{\jot}{5pt}
\left\{
\begin{alignedat}{2}
&\bvec{2C^TCx_e+2z+(A^T-I)\lambda}{2K^TKu_e+2v+B^T\lambda}+\nu^T\mu=0,\\
&Ax_e+Bu_e=x_e,\\
&\mu\geq 0,\\
&g(x_e,u_e)\mu=0,
\end{alignedat}
\right.
\end{equation}
where $\nu\in\R^{1\times(m+n)}$ is some subgradient of $g$ at $(x_e,u_e)$.

By applying \cite[Theorem 6.1 (ii)]{LarsGugdis} to the case where $z=0$, $v=0$ and $R$ is replaced by $\frac{R}{2}$, there exists a symmetric $P\in\R^{n\times n}$ and $s>0$ such that for any $(x,u)\in\R^{n+m}$
\begin{align*}
&(Ax+Bu)^TP(Ax+Bu)-x^TPx\\
&\qquad\qquad\leq \|Cx\|^2+\|Ku\|^2-s(\|x\|^2+\|u\|^2),
\end{align*}
where the $\|u\|^2$ term here follows easily from the positive definiteness of $R$. After substituting $x,u$ with $x-x_e,u-u_e$, and some straightforward calculations, we have that, for any $x\in \R^n$ and $u\in\R^m$,
%\begin{align*}
%&(Ax+Bu-Ax_e-Bu_e)^TP(Ax+Bu-Ax_e-Bu_e)\\
%&\qquad-(x-x_e)^TP(x-x_e)\\
%&\leq (x-x_e)^TC^TC(x-x_e)+(u-u_e)^TK^TK(u-u_e)\\
%&\qquad-s(\|x-x_e\|^2+\|u-u_e\|^2).
%\end{align*}
%Straightforward calculations show that
\begin{multline*}
(Ax+Bu)^TP(Ax+Bu)-2\<(A-I)x+Bu,Px_e\> +\\
-x^TPx \leq \ell(x,u)-\ell(x_e,u_e)-s(\|x-x_e\|^2+\|u-u_e\|^2)+ { }\\
-2\<x-x_e,C^TCx_e+z\>-2\<u-u_e,K^TKu_e+v\>.
\end{multline*}
%Besides, recall that
%\begin{equation*}
%\bvec{2C^TCx_e+2z+(A^T-I)\lambda}{2K^TKu_e+2v+B^T\lambda}+\nu^T\mu=0.
%\end{equation*}
Combining this inequality with the first identity in~\eqref{eq:pfthm}, we obtain that
\begin{align*}
&(Ax+Bu)^TP(Ax+Bu)-x^TPx + { } \\
&\qquad-\<(A-I)x+Bu,\lambda+2Px_e\>\\
&\leq \ell(x,u)-\ell(x_e,u_e)-s(\|x-x_e\|^2+\|u-u_e\|^2)+ { }\\
&\qquad+\left\<\nu^T\mu,\sbvec{x-x_e}{u-u_e}\right\>.
\end{align*}
Let $w=-\lambda-2Px_e$ and $V(x):=x^TPx+\<w,x\>$. Then
\begin{align}\label{precondis}
\begin{split}
&V(Ax+Bu)-V(x)\leq \ell(x,u)-\ell(x_e,u_e)+{}\\
&\qquad-s(\|x-x_e\|^2+\|u-u_e\|^2)+\mu\nu\sbvec{x-x_e}{u-u_e}
\end{split}
\end{align}
for any $(x,u)\in \R^{n+m}$. {From~\eqref{eq:pfthm} we know that $\nu$ is a subgradient of $g$ at $(x_e,u_e)$ and $\mu\geq0$}. Thus, we have that
\begin{equation*}
\mu\nu\bvec{x-x_e}{u-u_e}\leq \mu g(x,u)-\mu g(x_e,u_e).
\end{equation*}
Recall that $\mu g(x,u)\leq0$ if $(x,u)\in S$ and $\mu g(x_e,u_e)=0$ from the KKT conditions~\eqref{eq:pfthm}. So,
\begin{equation*}
\mu\nu\bvec{x-x_e}{u-u_e}\leq0,\;\;\forall (x,u)\in S.
\end{equation*}
Combining this with \eqref{precondis}, we obtain that
\begin{align*}
V(Ax+Bu)-V(x)\leq &\ \ell(x,u)-\ell(x_e,u_e)\\
&\quad-s(\|x-x_e\|^2+\|u-u_e\|^2),
\end{align*}
for any $(x,u)\in S$, which implies that problem $(CLQ)_T$ is strictly pre-dissipative at $(x_e,u_e)$ with storage function $V$.

Moreover, if the pair $(A,C)$ is detectable, then by \cite[Theorem 6.1 (i)]{LarsGugdis}, the matrix $P$ can be taken as positive definite. In this case, it turns out that the storage function $V$ is bounded from below, so the problem $(CLQ)_T$ is strictly dissipative at~$(x_e,u_e)$.
\end{proof}

\begin{rem}
{In fact, it was proved in \cite[Theorem 8.1]{LarsGugdis} (and \cite[Theorem 8.4]{LarsGugcont} for the continuous-time counterpart) that for unconstrained LQ OCP, the detectability of pair $(A,C)$ is also a necessary condition for strict dissipativity. This equivalence result can even be generalized to certain infinite dimensional problems, see, e.g.,~\cite{IFAC}.}
\end{rem}
\begin{rem}
{For a properly selected $g$, the last condition
\begin{equation*}
g(x_e,u_e)\mu=0
\end{equation*}
in ~\eqref{eq:pfthm} characterizes the nature of the turnpike reference. If we have $g(x_e,u_e)< 0$ and $\mu= 0$, then $(x_e,u_e)$ is an interior point of $S$, so $(x_e,u_e)$ is a local minimum of $\ell$ in $\R^n\times\R^m$. By the convexity of $\ell$, in fact $(x_e,u_e)$ is a global minimum of $\ell$ in $\R^n\times\R^m$. On the other hand, if $g(x_e,u_e)=0$ and $\mu\geq0$, then $(x_e,u_e)$ lies on the boundary of $S$.}
\end{rem}

{We shall use the following stabilizability assumption.
\begin{assumption}\label{assumptionA}
For any $x_0\in X_{\mathrm{tp}}\subset X$, there exists a control $u\in \calA^\infty(x_0)$ such that the corresponding trajectory~$x$ satisfies
\begin{equation}\label{traexp}
\|x(i)-x_e\|\leq M_0e^{-\rho i},\quad \forall i\in\N,
\end{equation}
where $(x_e,u_e)$ is the constrained optimal steady state and $M_0,\rho>0$ are constants independent of the selection of $x_0$.
\end{assumption}
}
\begin{lem}\label{lemexp}
Let Assumption \ref{assumptionA} be verified, then there exists some constant $M>0$ such that%, for any $x_0\in X_{\mathrm{tp}}$ and $N\in\N^+$,
\begin{equation}\label{keyestimate}
\sum_{i=0}^N \ell(x_N^*(i,x_0),u_N^*(i,x_0))-\ell(x_e,u_e)\leq M,
\end{equation}
for any $x_0\in X_{\mathrm{tp}}$ and $N\in\N^+$.
\end{lem}
\begin{proof}
{Let assumption~\ref{assumptionA} be satisfied with respect to $X_{\mathrm{tp}}\subset X$.} Fix some initial condition $x_0\in X_{\mathrm{tp}}$ and a control $u\in \calA^\infty(x_0)$ such that the corresponding trajectory $x(\cdot)$ satisfies \eqref{traexp}. Let $\mathbb{P}_M$ denote the projection operator to some subspace $M$. Define $\Omega:=\ker B$ and $\wt u:=\mathbb{P}_{\Omega^\perp}u+\mathbb{P}_{\Omega}u_e$. Since $B\mathbb{P}_{\Omega}u_e=0$, we have
\begin{align*}
x(i+1)&=Ax(i)+Bu(i)\\
&=Ax(i)+ B\wt u(i),\quad\forall i\in\N.
\end{align*}
We thus conclude that the trajectory corresponding to the initial condition $x_0$ and $\wt u$ coincides with $x$. Meanwhile, since $B$ is injective on $\Omega^\perp$, there exists some $m_0>0$ such that
\begin{equation*}
\|Bw\|\geq m_0\|w\|,\;\;\forall w\in\Omega^\perp.
\end{equation*}
Since $\wt u(i)-u_e=\mathbb{P}_{\Omega^\perp}(u(i)-u_e)\in\Omega^\perp$ for all $i\in\N$, we obtain that
\begin{align*}
\|\wt u(i)-u_e\|&\leq\frac{1}{m_0}\|B(\wt u(i)-u_e)\|\\
&=\frac{1}{m_0}\|x(i+1)-Ax(i)-x_e+Ax_e\|\\
&\leq\frac{1}{m_0}(\|x(i+1)-x_e\|+\|A\|\|x(i)-x_e\|),
\end{align*}
for any $i\in\N$. From \eqref{traexp}, we easily deduce that
\begin{equation}\label{conexp}
\!\|\wt u(i)-u_e\|\leq\frac{M_0e^{-\rho(i+1)}+\|A\|M_0e^{-\rho i}}{m_0}\leq M_1e^{-\rho i},
\end{equation}
for any $i\in\N$, for a suitable $M_1 > 0$.

On the other hand, for any $(x,u)\in\R^{n+m}$, simple calculations show that
\begin{align*}
\ell(x,u)-\ell(x_e,u_e)=&\<C^TCx+C^TCx_e+2z,x-x_e\>\\
&+\<K^TKu+K^TKu_e+2v,u-u_e\>,
\end{align*}
and thus
\begin{align}\label{costest}
\begin{split}
\ell(x,u)-&\ell(x_e,u_e)\\
\leq&\|C^TCx+C^TCx_e+2z\|\|x-x_e\|\\
&+\|K^TKu+K^TKu_e+2v\|\|u-u_e\|.
\end{split}
\end{align}
Combining \eqref{traexp} and \eqref{conexp}, we deduce that
\begin{equation*}
\|x(i)\|\leq \|x_e\|+M_0\;\;\text{and}\;\;\|u(i)\|\leq \|u_e\|+M_1,\;\forall i\in\N.
\end{equation*}
From \eqref{traexp} and \eqref{costest}, we thus conclude that
\begin{equation*}
\ell(x(i),u(i))-\ell(x_e,u_e)\leq M_2e^{-\rho i},\;\;\forall i\in\N,
\end{equation*}
where $M_2>0$ depends on $M_0$ and $M_1$, and
\begin{multline*}
\sum_{i=0}^N \ell(x_N^*(i,x_0),u_N^*(i,x_0))-\ell(x_e,u_e)\\
\leq \sum_{i=1}^N \ell(x(i),u(i))-\ell(x_e,u_e)\leq \frac{M_2}{1-e^{-\rho}},\quad \forall N\in\N.
\end{multline*}

Since the constants $M_0$ and $M_1$ do not depend on the selection of $x_0\in X_{\mathrm{tp}}$, neither does $M_2$. We thus conclude the estimate of the lemma with constant $M:=\frac{M_2}{1-e^{-\rho}}$.
\end{proof}

\begin{cor}\label{main2}
{Assume that the KKT conditions~\eqref{KKT} are satisfied at the constrained optimal steady state $(x_e,u_e)$. If Assumption \ref{assumptionA} and one of the following conditions hold:}
\begin{enumerate}
    \item [(a)] $X$ is bounded and, for any unobservable eigenvalue $\gamma$ of the pair $(A,C)$, we have that $|\gamma|\neq 1$.
    \item [(b)] The pair $(A,C)$ is detectable,
\end{enumerate}
then the measure turnpike property is satisfied at $(x_e,u_e)$.
\end{cor}
\begin{proof}
By Lemma \ref{lemexp}, there exists $M>0$ such that
\begin{equation*}
\sum_{i=0}^N \ell(x_N^*(i,x_0),u_N^*(i,x_0))-\ell(x_e,u_e)\leq M,
\end{equation*}
for any $x_0\in X_{\mathrm{tp}}$ and $N\in\N$. Besides, by Theorem \ref{dissres}, the problem~\eqref{cost}-\eqref{dynamic} is strictly dissipative at $(x_e,u_e)$, and there exist some storage function $V$ and sufficiently small number $s>0$ such that
\begin{multline*}
V(x_N^*(i+1,x_0))-V(x_N^*(i,x_0))\\
+s(\|x_N^*(i,x_0)-x_e\|^2+\|u_N^*(i,x_0)-u_e\|^2)\\
\leq \ell(x_N^*(i,x_0),u_N^*(i,x_0))-\ell(x_e,u_e),
\end{multline*}
where $x_0\in X_{\mathrm{tp}}$, $N\in\N^+$ and $i=0,\ldots,N-1$.
In particular, if our condition $(b)$ is satisfied, then $V$ is bounded from below. Summing on $i=0,\ldots,N-1$ and using the previous estimate, we easily obtain
\begin{multline*}
\displaystyle\Sigma_{i=0}^{N-1} s(\|x_N^*(i,x_0)-x_e\|^2+\|u_N^*(i,x_0)-u_e\|^2) +\\
V(x_N^*(N,x_0))-V(x_0) \leq M.
\end{multline*}
We notice that the term $V(x_N^*(N,x_0))-V(x_0)$ is uniformly bounded with respect to $x_0\in X_{\mathrm{tp}}$ and $N\in\N^+$. Consequently, there exists a constant $M_E$ such that for any $x_0\in X_{\mathrm{tp}}$ and $N\in\N^+$,
\begin{equation*}
\Sigma_{i=0}^{N-1} \|x_N^*(i,x_0)-x_e\|^2+\|u_N^*(i,x_0)-u_e\|^2\leq M_E.
\end{equation*}
For any $\eps>0$, by considering only the terms larger than $\eps$, we notice that, for any $x_0\in X_{\mathrm{tp}}$ and $N\in\N^+$,
\begin{equation*}
\#_{x_0,N}\eps\leq \Sigma_{i=0}^{N-1} \|x_N^*(i,x_0)-x_e\|^2+\|u_N^*(i,x_0)-u_e\|^2,
\end{equation*}
where we have abbreviated the number
\begin{align*}
\#\{ i = 0,\ldots,N-1 \; | \; \| x^*_N&(i,x_0)-x_e\|^2\\
&+\| u^*_N(i,x_0)-u_e\|^2 >\eps\}
\end{align*}
as $\#_{x_0,N}$. Combining the two previous estimates, we get
\begin{equation*}
\#_{x_0,N}\leq \frac{M_E}{\eps},\quad\forall x_0\in X_{\mathrm{tp}},\;N\in\N^+.
\end{equation*}
Hence we conclude that problem~\eqref{cost}-\eqref{dynamic} satisfies the measure turnpike property on the initial set $X_{\mathrm{tp}}$ at $(x_e,u_e)$. 
\end{proof}
{We now illustrate how we can use our results to ensure the turnpike property in some applications. The first example shows that the measure turnpike property may only be satisfied on certain initial set $X_{\mathrm{tp}}$. 
\begin{example}
Let $S$ and $X$ be defined by
\begin{equation*}
S:=\{(x,u)\in\R^{2+2}\,|\,\|x\|_{\infty}\leq 1,\,\|u\|_{\infty}\leq 0.1\},
\end{equation*}
$$
X=\{x\in\R^{2}\,|\,\|x\|_{\infty}\leq 1\}.
$$ 
In this case, $g(x,u)=\max\{\|x\|_{\infty} -1,\|u\|_{\infty} -0.1\}$ for all $(x,u)\in\R^{2+2}$. Let 
\begin{align*}
&A=e^{-0.1}\sbmat{\cos\frac{\pi}{4}}{\sin\frac{\pi}{4}}{-\sin\frac{\pi}{4}}{\cos\frac{\pi}{4}},\quad B=\sbmat{1}{0}{0}{1},\\
&C=\sbmat{0}{0}{0}{0},\quad K=\sbmat{1}{0}{0}{1},\quad z=v=\sbvec{0}{0}.
\end{align*}
Clearly, the optimal steady state is $(x_e,u_e)=(0,0)\in\R^{2+2}$. We claim that the points sufficiently close to the vertices of the square $X$ cannot be stabilized to $x_e=0$. In fact, if we take $x_0\approx[1\; 1]^T$ and any admissible $u(0)$, then
\begin{equation*}
x(1)\approx e^{-0.1}\sbvec{\cos\frac{\pi}{4}+\sin\frac{\pi}{4}}{-\sin\frac{\pi}{4}+\cos\frac{\pi}{4}}+u(0)=\sbvec{e^{-0.1}\sqrt{2}}{0}+u(0).
\end{equation*}
So, $\|x(1)\|_{\infty}\approx e^{-0.1}\sqrt{2}+(u(0))_1\geq 1.18$, which implies $x(1)\notin X$. 

However, if we let $X_{\mathrm{tp}}:=\{x\in\R^2\,|\,\|x\|\leq1\}$, simple calculations show that the pair $(x^*,u^*)$ (which is actually the optimal pair) given by
\begin{align*}
\begin{split}
&x^*(\tau)=e^{-0.1 \tau}\sbmat{\cos\frac{\pi \tau}{4}}{\sin\frac{\pi \tau}{4}}{-\sin\frac{\pi \tau}{4}}{\cos\frac{\pi \tau}{4}}x_0,\\
&u^*(\tau)\equiv0,\;\;\;\forall \tau\in\N,
\end{split}
\end{align*}
with $x_0\in X_{\mathrm{tp}}$ satisfies that
\begin{equation*}
\|x^*(\tau)\|=e^{-0.1 \tau}\|x_0\|\leq 1,\quad\forall \tau\in\N,
\end{equation*}
so we must have $(x^*(\tau),u^*(\tau))\in S$ for all $\tau\in\N$, which verifies our Assumption \ref{assumptionA}. Besides, since {$(x_e,u_e)$} is an interior steady state in $S$, so Slater's condition is satisfied and KKT conditions~\eqref{KKT} are verified at $(x_e,u_e)$. Finally, we can easily verify that the pair $(A,C)$ is detectable (see, e.g., \cite[Lemma 7.1.2]{Sontag}). Thus, by Corollary \ref{main2}, our problem satisfies the measure turnpike property on the initial set $X_{\mathrm{tp}}$.
\end{example}
\begin{rem}\label{rm:afterEx}
In \cite{LarsGugdis}, it was shown that for unconstrained problems the measure turnpike property is satisfied (on arbitrary compact set) if and only if the pair $(A,B)$ is stabilizable and the pair $(A,C)$ is detectable. The above example shows that a similar characterization of the measure turnpike (on certain initial set) solely based on spectral conditions like stabilizability and detectability as in the unconstrained case cannot hold for constrained problems, because the turnpike may depend on the selection of the initial set $X_{\mathrm{tp}}$.
\end{rem}}
The next example illustrates the occurrence of the measure turnpike at a boundary turnpike reference.%, as described in Corollary \ref{main2}.
\begin{example}
Let $S$ be defined by
\begin{equation*}
S:=\left\{(x,u)\in\R^{3+1}\,|\,x_3\geq \sqrt{x_1^2+x_2^2},\;u\in[-1,1]\right\}.
\end{equation*}
Notice that $S$ is described as the sublevel set of the function $g(x,u)=\max\{\sqrt{x_1^2+x_2^2}-x_3,u^2-u\}$, for all $(x,u)\in\R^{3+1}$, and
$$
X=\left\{x\in\R^{3}\,|\,x_3\geq \sqrt{x_1^2+x_2^2}\right\}.
$$
In fact, the set $X$ is a cone with vertex at $(0,0,0)$.
Let us now fix $X_{\mathrm{tp}}$ to be an arbitrary bounded subset of $X$ and
\begin{align*}
&A=\left[\begin{smallmatrix}
e^{-0.1}\cos\frac{\pi}{4}&e^{-0.1}\sin\frac{\pi}{4}&0\\-e^{-0.1}\sin\frac{\pi}{4}&e^{-0.1}\cos\frac{\pi}{4}&0\\0&0&1
\end{smallmatrix}\right],\quad B=\left[\begin{smallmatrix}
0\\0\\-1
\end{smallmatrix}\right]\\
&C=\left[\begin{smallmatrix}
0&0&0\\
0&0&0\\
0&0&1
\end{smallmatrix}\right],\quad
K=1.
\end{align*}
Besides, set $z=[0\ 0\ 1]^T$ and $v=1$. 
Simple calculations show that, for any $(x,u)\in\R^{3+1}$,
\begin{equation*}
\ell(x,u):=(x_3+1)^2+(u+1)^2-2.
\end{equation*}
So, the global optimal steady state in $\R^{3+1}$ is given by
\begin{equation*}
(x_{\mathrm{g}},u_{\mathrm{g}})=\left(\left[
\begin{smallmatrix}
0\\0\\-1
\end{smallmatrix}\right], 0\right).
\end{equation*}
As we can see, this steady state (which does not belong to $S$) is different from the constrained optimal steady state
\begin{equation*}
(x_e,u_e)=\left(\left[
\begin{smallmatrix}
0\\0\\0
\end{smallmatrix}\right],0
\right).
\end{equation*}
{Moreover, since the additional steady state
\begin{equation*}
(x_{\iso},u_{\iso})=\left(\left[
\begin{smallmatrix}
0\\0\\1
\end{smallmatrix}\right],0
\right).
\end{equation*}
is in the interior of $S$, Slater's condition is satisfied and thus the KKT conditions~\eqref{KKT} are verified at $(x_e,u_e)$.} On the other hand, it is easy to check that the pair $(A,C)$ is detectable. So, to apply Corollary \ref{main2}, we only need to verify the exponential stablizability of the initial set $X_{\mathrm{tp}}$ to $(x_e,u_e)$ in~\eqref{traexp}. Indeed, we claim that the feedback control defined by
\begin{equation}\label{law}
u(x)=\min\left\{1,x_3-e^{-0.1}\sqrt{x_1^2+x_2^2}\right\},\quad\forall x\in\R^3,
\end{equation}
stabilizes $X_{\mathrm{tp}}$ to $(x_e,u_e)$ at a uniform exponential rate. First, notice that $u(x)\in[0,1]$ if $x\in X$. Now assume that, for some $N\in\N$, the trajectory corresponding to initial condition $x_0\in X_{\mathrm{tp}}$ and control \eqref{law} satisfies $(x(i),u(i))\in S$, for any $i=0,\ldots,N$. Then
\begin{align*}
\begin{split}
x(N+1)&=Ax(N)+Bu(N)\\
&=\left[
\begin{smallmatrix}
e^{-0.1}\cos\frac{\pi}{4}x_1(N)+e^{-0.1}\sin\frac{\pi}{4}x_2(N)\\-e^{-0.1}\sin\frac{\pi}{4}x_1(N)+e^{-0.1}\cos\frac{\pi}{4}x_2(N)\\x_3(N)-u(x(N)).
\end{smallmatrix}\right].
\end{split}
\end{align*}
Observe that
\begin{equation}\label{stablepart}
\!\!\!\!\!\sqrt{x_1^2(N+1)+x_2^2(N+1)}=e^{-0.1}\sqrt{x_1^2(N)+x_2^2(N)}\,.
\end{equation}
So,
\begin{multline*}
x_3(N+1)-\sqrt{x_1^2(N+1)+x_2^2(N+1)}\\
=x_3(N)-e^{-0.1}\sqrt{x_1^2(N)+x_2^2(N)}-u(x(N))\geq 0,
\end{multline*}
and thus $(x(N+1),u(x(N+1)))\in S$. By induction, $u\in\calA^\infty(x_0)$. 
On the other hand, since $X_{\mathrm{tp}}$ is bounded, there exists some $M>0$ such that
\begin{equation}\label{boundx3}
(x_0)_3\leq M,\quad \forall x_0\in X_{\mathrm{tp}}.
\end{equation}
Besides, notice that $x_3(N+1)=(x_0)_3-\Sigma_{i=0}^{N}u(x(i))$ for all $N\in\N$. 
As a consequence of \eqref{law} and \eqref{boundx3}, after at most $M$ steps, the term $x_3-\sqrt{x_1^2+x_2^2}$ would become smaller than $1$. So,
\begin{equation*}
u(x(M))=x_3(M)-e^{-0.1}\sqrt{x_1^2(M)+x_2^2(M)}
\end{equation*}
and
\begin{align*}
x_3(M+1)&=x_3(M)-x_3(M)+e^{-0.1}\sqrt{x_1^2(M)+x_2^2(M)}\\
&=e^{-0.1}\sqrt{x_1^2(M)+x_2^2(M)},
\end{align*}
and therefore by \eqref{stablepart} we have that
\begin{equation*}
x_3(M+1)-\sqrt{x_1^2(M+1)+x_2^2(M+1)})=0.
\end{equation*}
Again by induction, we obtain that
\begin{equation*}
x_3(N+1)=e^{-0.1}\sqrt{x_1^2(N)+x_2^2(N)},\;\;\forall N\in\N,\ N\geq M.
\end{equation*}

Finally, the above estimate and \eqref{stablepart} imply that the initial set $X_{\mathrm{tp}}$ can be stabilized to $(x_e,u_e)$ in uniform exponential rate. Hence, by Corollary \ref{main2}, the measure turnpike property is satisfied on $X_{\mathrm{tp}}$. 
\end{example}

\bibliographystyle{LCSS}
\bibliography{references.bib}

\end{document}